\documentclass{amsart}
\usepackage{amsmath}
\usepackage{amsfonts}
\newcommand{\N}{\mathbb{N}}
\newcommand{\R}{\mathbb{R}}

\newtheorem{theorem}{Theorem}
\theoremstyle{plain}

\newtheorem{corollary}{Corollary}

\newtheorem{lemma}{Lemma}

\numberwithin{equation}{section}
\input{tcilatex}

\begin{document}
\title[\'{C}iri\'{c}'S Fixed Point Theorem in a cone metric space]{\'{C}iri\'{c}'S Fixed Point Theorem in a cone metric space}
\author{Bessem Samet}
\address{DEPARTMENT OF MATHEMATICS, TUNIS COLLEGE OF SCIENCES AND TECHNIQUES, 5 AVENUE TAHA HUSSEIN, BP, 59, BAB MANARA, TUNIS.}
\email{bessem.samet@gmail.com}
\begin{abstract}
In this paper, we extend a fixed point theorem  due to \'{C}iri\'{c} to a cone metric space.
\end{abstract} \maketitle
\section{Introduction and preliminaries}
Many generalizations of the Banach contraction principle \cite{BANACH} have been considered in the literature (see \cite{A1}-\cite{beg}, \cite{beg2}-\cite{sam}).

Huang and Zhang \cite{Cone} recently have introduced the concept of cone metric space, where the set of real numbers is replaced by an ordered Banach space, and they have established
some fixed point theorems for contractive type mappings in a normal cone metric space. The study of fixed point theorems in such spaces is followed
by some other mathematicians (see \cite{A1}-\cite{beg}, \cite{beg2}, \cite{A2}, \cite{J}, \cite{rit}).

In this paper, we extend  a fixed point theorem  due to \'{C}iri\'{c} (\cite{cir}-Theorem 2.5) to a cone metric space.
Before presenting our result, we start by recalling some definitions.

Let $E$ be a real Banach space and $P$ a subset of $E$. $P$ is called a cone if and only
if:
\begin{itemize}
\item[(i)] $P$ is closed, nonempty, and $P \neq \{0\}$.
\item[(ii)] $a, b \in \mathbb{R}$, $a, b \geq 0$,  $x,y \in P \Rightarrow ‖x + by ―in  P$.
\item[(iii)] $x \in  P$ and $-x \in  P$  $\Rightarrow x = 0$.
\end{itemize}

Given a cone $P \subset E$, we define a partial ordering $\leq $ with respect to $P$ by:
$$
x \leq y  \Leftrightarrow  y - x \in  P.
$$
We shall write $x <y$ to indicate that $x \leq  y$ but $x \neq y$, while $x \ll y$ will
stand for $y - x ―in  int P$ , where $int P$ denotes the interior of $P$.

The cone $P$ is called normal if there is a number $k >0$ such that for all $x,y \in  E$,
$$
0 \leq x \leq  y \Rightarrow  \|x\|\leq k \|y\|,
$$
where $\|\cdot\|$ is the norm in $E$. In this case, the number $k$ is called the normal constant of $P$.
Rezapour and Hamlbarani \cite{rit} proved that there are no normal cones with normal constant $k < 1$ and for
each $c > 1$ there are cones with normal constant $k > c$. For this reason, in all this paper, we take $k\geq 1$.

In the following we always suppose $E$ is a Banach space, $P$ is a cone in $E$ with $\mbox{ int }P \neq \emptyset$ and $\leq $ is partial ordering with respect to $P$.
As it has been defined in \cite{Cone}, a function $d : X ―times X ―rightarrow E$ is
called a cone metric on $X$ if it satisfies the following conditions:

\begin{itemize}
\item[(a)]  $0<d(x,y)$ for all $x,y ―in  X$, $x\neq y$  and $d(x, y) = 0$ if and only if $x = y$.
\item[(b)] $d(x, y) = d(y, x)$ for all $x,y ―in  X$.
\item[(c)] $d(x, y) \leq  d(x, z) + d(z,y)$ for all $x, y, z \in X$.
\end{itemize}
Then $(X, d)$ is called a cone metric space.

Let $(x_n)$ be a sequence in $X$ and $x\in X$.
\begin{itemize}
\item If for every $c\in E$, $c\gg0 $  there is $N$ such that for all $n>N$, $d(x_n, x)\ll c$, then $(x_n)$ is
said to be convergent to $x$ and $x$ is the limit of $(x_n)$. We denote this by $x_n\rightarrow x \mbox{ as } n\rightarrow +\infty$.
 \item If for any $c \in E$ with $0 \ll c$, there is $N$ such that for all $n,m > N$, $d(x_n, x_m) \ll c$, then $(x_n)$ is called a
Cauchy sequence in $X$.
\end{itemize}

Let $(X,d)$ be a cone metric space. If every Cauchy sequence is convergent in $X$, then $X$ is called a complete cone metric space.

The following lemmas  will be  useful later.
\begin{lemma}(Huang and Zhang \cite{Cone})
Let $(X, d)$ be a cone metric space, $P$ be a normal cone.
Let $(x_n)$ be a sequence in $X$. Then $(x_n)$ converges to $x$ if and only if $\|d(x_n, x)\|\rightarrow 0$ as $n\rightarrow+\infty$.
\end{lemma}

\begin{lemma}(Huang and  Zhang \cite{Cone})
Let $(X, d)$ be a cone metric space, $(x_n)$ be a sequence in $X$. If $(x_n)$ is convergent, then it is a Cauchy sequence, too.
\end{lemma}

\begin{lemma}(Huang and Zhang \cite{Cone})
Let $(X, d)$ be a cone metric space, $P$ be a normal cone. Let $(x_n)$ be a sequence in $X$.
Then, $(x_n)$ is a Cauchy sequence if and only if $\|d(x_n,x_m)\|\rightarrow 0$ as $n,m \rightarrow +\infty$.
\end{lemma}

We denote $\mathcal{L}(E)$ the set of  linear bounded operators on $E$, endowed with the following norm:
$$
\|S\|=\sup_{x\in E, x\neq 0} \frac{\|Sx\|}{\|x\|},\,\,\forall\,S\in \mathcal{L}(E).
$$
It is clear  that if $S\in \mathcal{L}(E)$, we have:
$$
\|Sx\|\leq \|S\|\|x\|,\,\,\forall\,x\in E.
$$
We denote by $I: E\rightarrow E$ the identity operator, i.e., $Ix=x,\,\forall\,x\in X$.
If $S\in \mathcal{L}(E)$, we denote by $S^{-1}\in \mathcal{L}(E)$ (if such operator exists) the operator defined by:
$$
S^{-1}Sx=SS^{-1}x=x,\,\,\forall\,x\in E.
$$

\section{Fixed point theorem}
The main result of this paper is the following.
\begin{theorem}\label{T}
Let $(X, d)$ be a complete cone metric space, $P$ be a normal cone with normal constant $k$ ($k\geq 1$).
Suppose the mapping $T: X\rightarrow X$ satisfies the following contractive condition:
\begin{eqnarray}
\label{I} d(Tx,T y)&\leq & A_1(x,y)d(x,y)+A_2(x,y)d(x,Tx)+A_3(x,y)d(y,Ty)\\
\nonumber && +A_4(x,y)d(x,Ty) +A_4(x,y)d(y,Tx),
\end{eqnarray}
for all $x, y\in X$, where $A_i: X\times X \rightarrow \mathcal{L}(E)$, $i=1,\cdots,4$.
Further, assume that for all $x, y\in X$, we have: 
\begin{eqnarray}
\label{i1} \exists\, \alpha\in [0,1/k)\,|\, \sum_{i=1}^4 \|A_i(x,y)\|+\|A_4(x,y)\| \leq \alpha \\
\label{i2}\exists\, \beta\in [0,1)\,|\,\|S(x,y)\|\leq \beta\\
\label{i3} (A_1(x,y)+A_2(x,y))(P)\subseteq P\\
\label{hb} A_2(x,y)(P)\subseteq P\\
\label{i4} A_4(x,y)(P)\subseteq P\\
\label{i5} (I-A_3(x,y)-A_4(x,y))^{-1}(P)\subseteq P.
\end{eqnarray}
Here, $S: X\times X \rightarrow \mathcal{L}(E)$ is given by:
$$
S(x,y)=(I-A_3(x,y)-A_4(x,y))^{-1} (A_1(x,y)+A_2(x,y)+A_4(x,y)),\,\,\forall\,x,y\in X.
$$
Then, $T$ has a unique fixed point.
\end{theorem}
\begin{proof} Let $x\in X$ be arbitrary and define the sequence $(x_n)_{n\in \N}\subset X$ by:
$$
x_0=x, x_1=Tx_0, \cdots, x_n=Tx_{n-1}=T^nx_0,\cdots
$$
By (\ref{I}), we get:
\begin{eqnarray*}
d(x_n,x_{n+1})&=&d(Tx_{n-1},Tx_n)\\
&\leq &  A_1(x_{n-1},x_n)d(x_{n-1},x_n)+A_2(x_{n-1},x_n)d(x_{n-1},x_n)\\
&& +A_3(x_{n-1},x_n)d(x_n,x_{n+1})+A_4(x_{n-1},x_n)d(x_{n-1},x_{n+1})\\
&& +A_4(x_{n-1},x_n)d(x_n,x_n)\\
&=& (A_1(x_{n-1},x_n)+A_2(x_{n-1},x_n))d(x_{n-1},x_n)+A_3(x_{n-1},x_n)d(x_n,x_{n+1})\\
&& +A_4(x_{n-1},x_n) d(x_{n-1},x_{n+1}).
\end{eqnarray*}

Using the triangular inequality, we get:
$$
d(x_{n-1},x_{n+1})\leq d(x_{n-1},x_{n})+d(x_{n},x_{n+1}),
$$
i.e.,
$$
d(x_{n-1},x_{n})+d(x_{n},x_{n+1})-d(x_{n-1},x_{n+1})\in P.
$$
From (\ref{i4}), it follows that:
$$
A_4(x_{n-1},x_n)[d(x_{n-1},x_{n})+d(x_{n},x_{n+1})-d(x_{n-1},x_{n+1})]\in P,
$$
i.e.,
$$
A_4(x_{n-1},x_n) d(x_{n-1},x_{n+1})\leq A_4(x_{n-1},x_n) d(x_{n-1},x_{n})+A_4(x_{n-1},x_n) d(x_{n},x_{n+1}).
$$
Then, we have:
\begin{eqnarray*}
d(x_n,x_{n+1})&\leq & (A_1(x_{n-1},x_n)+A_2(x_{n-1},x_n)+A_4(x_{n-1},x_n))d(x_{n-1},x_n)\\
&&+ (A_3(x_{n-1},x_n)+A_4(x_{n-1},x_n))d(x_n,x_{n+1}).
\end{eqnarray*}
Hence,
\begin{eqnarray*}
(I-A_3(x_{n-1},x_n)-A_4(x_{n-1},x_n))d(x_n,x_{n+1})&\leq &(A_1(x_{n-1},x_n)+A_2(x_{n-1},x_n)\\
&&+A_4(x_{n-1},x_n))d(x_{n-1},x_n).
\end{eqnarray*}
Using (\ref{i5}), we get:
\begin{equation}\label{hes}
d(x_n,x_{n+1})\leq S(x_{n-1},x_n)d(x_{n-1},x_n).
\end{equation}
It is not difficult to see that under hypotheses (\ref{i3}), (\ref{i4}) and (\ref{i5}), we have:
$$
S(x,y)(P)\subseteq P,\,\,\forall\,x,y\in X.
$$
Using this remark, (\ref{hes}) and proceeding by iterations, we get:
$$
d(x_n,x_{n+1})\leq S(x_{n-1},x_n) S(x_{n-2},x_{n-1})\cdots S(x_0,x_1)d(x_0,x_1),
$$
which implies by (\ref{i2}) that:
$$
\|d(x_n,x_{n+1})\|\leq k  \|S(x_{n-1},x_n)\|  \|S(x_{n-2},x_{n-1})\|\cdots \|S(x_0,x_1)\| \|d(x_0,x_1)\|\leq k \beta^n\|d(x_0,x_1)\|.
$$
For any positive integer $p$, we have:
$$
d(x_n,x_{n+p}) \leq \sum_{i=1}^p d(x_{n+i-1},x_{n+i}),
$$
which implies that:
\begin{eqnarray}
\nonumber\|d(x_n,x_{n+p})\|&\leq& k \sum_{i=1}^p \|d(x_{n+i-1},x_{n+i})\|\\
\nonumber &\leq &   k^2 \sum_{i=1}^p \beta^{n+i-1} \|d(x_0,x_1)\|\\
\label{mab}&\leq & k^2\frac{\beta^n}{1-\beta} \|d(x_0,x_1)\|.
\end{eqnarray}
Since $\beta\in [0,1)$, $\beta^n\rightarrow 0$ as $n\rightarrow +\infty$. So from (\ref{mab}) it follows
that the sequence $(x_n)_{n\in \N}$ is Cauchy. Since $(X,d)$ is complete, there is a point $u\in X$ such that:
\begin{equation}\label{saber}
\lim_{n\rightarrow +\infty} d(Tx_n,u)=\lim_{n\rightarrow +\infty} d(x_n,u)=\lim_{n\rightarrow +\infty}d(x_n,x_{n+1})=0.
\end{equation}
Now, using the contractive condition (\ref{I}), we get:
\begin{eqnarray*}
d(Tu,Tx_n)&\leq& A_1(u,x_n)d(u,x_n)+A_2(u,x_n)d(u,Tu)\\
&& +A_3(u,x_n)d(x_n,x_{n+1})+A_4(u,x_n)d(u,x_{n+1})\\
&& +A_4(u,x_n)d(x_n,Tu).
\end{eqnarray*}
By the triangular inequality, we have:
\begin{eqnarray*}
d(u,Tu)&\leq& d(u,x_{n+1})+d(x_{n+1},Tu)\\
d(x_n,Tu) &\leq & d(x_n,Tx_n)+d(Tx_n,Tu).
\end{eqnarray*}
By (\ref{hb}) and (\ref{i4}), we get:
\begin{eqnarray*}
A_2(u,x_n)d(u,Tu)&\leq& A_2(u,x_n)(d(u,x_{n+1})+d(x_{n+1},Tu))\\
A_4(u,x_n)d(x_n,Tu) &\leq& A_4(u,x_n)d(x_n,Tx_n)+A_4(u,x_n)d(Tx_n,Tu).
\end{eqnarray*}
Hence,
\begin{eqnarray*}
d(Tu,Tx_n)&\leq&A_1(u,x_n)d(u,x_n)+(A_2(u,x_n)+A_4(u,x_n))d(u,x_{n+1})\\
&&+ (A_2(u,x_n)+A_4(u,x_n))d(x_{n+1},Tu)\\&&+(A_3(u,x_n)+A_4(u,x_n))d(x_n,x_{n+1}).
\end{eqnarray*}
Using (\ref{i1}), this inequality implies that:
$$
\|d(Tu,Tx_n)\|\leq \frac{k\alpha}{1-k\alpha} (\|d(u,x_n)\|+\|d(u,x_{n+1})\|+\|d(x_n,x_{n+1})\|).
$$
From (\ref{saber}), it follows immediately that:
\begin{equation}\label{kh}
\lim_{n\rightarrow +\infty} d(Tu,Tx_n)=0.
\end{equation}
Then, (\ref{saber}), (\ref{kh}) and the uniqueness of the limit imply that $u=Tu$, i.e., $u$ is a fixed point of $T$.
So we proved that $T$ has least one fixed point $u\in X$.

Now, if $v\in X$ is another fixed point of $T$, by (\ref{I}), we get:
$$
d(u,v)=d(Tu,Tv)\leq A_1(u,v)d(u,v)+2A_4(u,v)d(u,v),
$$
which implies that:
$$
\|d(u,v)\|\leq k(\|A_1(u,v)\|+2\|A_4(u,v)\|)\|d(u,v)\|\leq k\alpha \|d(u,v)\|,
$$
i.e.,
$$
(1-k\alpha)\|d(u,v)\|\leq 0.
$$
Since $0\leq \alpha<1/k$, we get $d(u,v)=0$, i.e., $u=v$. So the proof of the theorem is complete.
\end{proof}

Now, we will show that Theorem 2.5 of \'{C}iri\'{c} \cite{cir} is a particular case of Theorem \ref{T}.
\begin{corollary}
Let $(X, d)$ be a complete metric space and $T: X\rightarrow X$ be a mapping satisfying the following contractive condition:
\begin{eqnarray}
\label{cc} d(Tx,Ty)&\leq & a_1(x,y)d(x,y)+a_2(x,y)d(x,Tx)+a_3(x,y)d(y,Ty)\\
\nonumber &&+a_4(x,y)(d(x,Ty)+d(y,Tx)),
\end{eqnarray}
for all $x,y\in X$, where $a_i: X\times X\rightarrow [0, +\infty)$, $i=1,\cdots,4$ and
$\displaystyle\sum_{i=1}^4\alpha_i(x,y)+\alpha_4(x,y)\leq \alpha$ for each $x, y\in X$ and some $\alpha\in [0,1)$.
Then, $T$ has a unique fixed point.
\end{corollary}
\begin{proof}
We take $E=\R$ (with the usual norm) and $P=[0,+\infty)$. Then, $(X,d)$ is a complete cone metric space and $P$ is a normal cone with
normal constant $k=1$. For each $i=1,\cdots,4$, we define $A_i: X\times X\rightarrow \mathcal{L}(E)$ by:
$$
A_i(x,y): t\in \R\mapsto a_i(x,y)t,
$$
for all $x,y\in X$. let us check now that all the required hypotheses of Theorem \ref{T} are satisfied.
\begin{itemize}
\item[$\bullet$] Condition (\ref{cc}) implies that:
\begin{eqnarray*}
d(Tx,T y)&\leq & A_1(x,y)d(x,y)+A_2(x,y)d(x,Tx)+A_3(x,y)d(y,Ty)\\
&& +A_4(x,y)d(x,Ty) +A_4(x,y)d(y,Tx),
\end{eqnarray*}
for all $x, y\in X$. Then, condition (\ref{I}) of Theorem \ref{T} is satisfied.
\item[$\bullet$] For all $i=1,\cdots,4$, we have:
$$
\|A_i(x,y)\|=a_i(x,y),\,\,\forall\,x,y\in X.
$$
Then,
$$
\sum_{i=1}^4 \|A_i(x,y)\|+\|A_4(x,y)\| \leq \alpha,\,\,\forall\,x, y\in X
$$
and condition (\ref{i1}) of Theorem \ref{T} is satisfied.
\item[$\bullet$] For all $x,y\in X$, we have:
$$
S(x,y)t =\frac{a_1(x,y)+a_2(x,y)+a_4(x,y)}{1-a_3(x,y)-a_4(x,y)}\,t,\,\,\forall\,t\in \R.
$$
Then, for all $x, y\in X$, we have:
$$
\|S(x,y)\|=\frac{a_1(x,y)+a_2(x,y)+a_4(x,y)}{1-a_3(x,y)-a_4(x,y)}.
$$
Since $\alpha\in [0,1)$, we have:
$$
a_1(x,y)+a_2(x,y)+a_4(x,y)+\alpha a_3(x,y)+\alpha a_4(x,y)\leq \alpha,\,\,\forall\,x,y\in X.
$$
Then,
$$
\|S(x,y)\|\leq \alpha,\,\,\forall\,x,y\in X
$$
and condition (\ref{i2}) of Theorem \ref{T} holds with $\beta=\alpha$.
\item[$\bullet$] Conditions (\ref{i3}), (\ref{hb}) and (\ref{i4}) are easy to check.
\item[$\bullet$] For all $x,y\in X$, we have:
$$
(I-A_3(x,y)-A_4(x,y))^{-1}s= \frac{s}{1-a_3(x,y)-a_4(x,y)},\,\,\forall\,s\in \R.
$$
Since $a_3(x,y)+a_4(x,y)< 1$ for all $x,y\in X$, then
$$
s\geq 0 \Rightarrow (I-A_3(x,y)-A_4(x,y))^{-1}s\geq 0.
$$
Hence, condition (\ref{i5}) of Theorem \ref{T} is satisfied.
\end{itemize}
Now, we are able to apply Theorem \ref{T} and then, $T$ has a unique fixed point.
\end{proof}

\section{Open problem}
We present the following open problem.

In hypothesis (\ref{i1}), we assumed that $\alpha \in [0, 1/k)$, where $k$ is the normal constant of the cone $P$.
What can we say about the case when $\alpha\in [1/k,1)$ with $k>1$?



\bibliographystyle{9}

\end{document}